\documentclass[12pt,a4paper,reqno]{amsart}
\usepackage{amssymb}
\usepackage{amscd}
\numberwithin{equation}{section}

\theoremstyle{plain}

\newtheorem{theorem}[subsection]{Theorem}
\newtheorem{proposition}[subsection]{Proposition}
\newtheorem{lemma}[subsection]{Lemma}

\theoremstyle{definition}

\newtheorem{definition}[subsection]{Definition}

\newtheorem{remark}[subsection]{Remark}

\renewcommand{\leq}{\leqslant}
\renewcommand{\geq}{\geqslant}

\newsavebox{\proofbox}
\savebox{\proofbox}{\begin{picture}(7,7)%
  \put(0,0){\framebox(7,7){}}\end{picture}}

\newcommand\E{\mathbb{E}}
\newcommand\Z{\mathbb{Z}}
\newcommand\R{\mathbb{R}}

\newcommand\N{\mathbb{N}}

\newcommand\eps{\varepsilon}

\renewcommand\mod{{\ \operatorname{mod}\ }}

\begin{document}

\title[Multi-dimensional Szemer\'edi via correspondence]{A multi-dimensional Szemer\'edi theorem for the primes via a correspondence principle}

\author{Terence Tao}
\address{UCLA Department of Mathematics, Los Angeles, CA 90095-1596.
}
\email{tao@math.ucla.edu}

\author{Tamar Ziegler}
\address{Department of Mathematics, Technion, Haifa, Israel 32000.}
\email{tamarzr@tx.technion.ac.il}

\thanks{The first author is supported by NSF grant DMS-0649473 and by a Simons Investigator Award. The second author is supported by ISF grant 407/12. The second author was on sabbatical at Stanford at the time this work was carried out; she would like to thank the Stanford math department for its hospitality and support. }


\begin{abstract}  We establish a version of the Furstenberg-Katznelson multi-dimensional  Szemer\'edi theorem in the primes ${\mathcal P} := \{2,3,5,\ldots\}$, which roughly speaking asserts that any dense subset of ${\mathcal P}^d$ contains finite constellations of any given rational shape.  Our arguments are based on a weighted version of the Furstenberg correspondence principle, relative to a weight which obeys an infinite number of pseudorandomness (or ``linear forms'') conditions, combined with the main results of a series of papers by Green and the authors which establish such an infinite number of pseudorandomness conditions for a weight associated with the primes. The same result, by a rather different method, has been simultaneously established by Cook, Magyar, and Titichetrakun, and more recently by Fox and Zhao.
\end{abstract}

\maketitle

\section{Introduction}

We recall the following famous theorem of Szemer\'edi:

\begin{theorem}[Szemer\'edi's theorem]\label{sz-thm}\cite{szemeredi}  Let $k \geq 1$ be a natural number, and let $\delta > 0$.  Then if $N$ is sufficiently large, every subset $A$ of $[N] := \{ n \in \Z: 1 \leq n \leq N \}$ of cardinality $|A| \geq \delta N$ contains an arithmetic progression $a,a+r,\ldots,a+(k-1)r$ of length $k$, where $a \in \Z$ and $r$ is a positive integer.
\end{theorem}

In \cite{fk0}, Furstenberg and Katznelson established a multi-dimensional generalisation of Szemer\'edi's theorem:

\begin{theorem}[Multidimensional Szemer\'edi theorem]\label{fk-thm}\cite[Theorem B]{fk0}  Let $d \geq 1$ be a natural number, let $v_1,\ldots,v_k$ be elements of $\Z^d$, and let $\delta > 0$.  Then if $N$ is sufficiently large, every subset $A$ of $[N]^d$ of cardinality $|A| \geq \delta N^d$ contains a set of the form $a+rv_1,\ldots,a+rv_k$, where $a \in \Z^d$ and $r$ is a positive integer.
\end{theorem}

The pattern $a+rv_1,\ldots,a+rv_k$ is sometimes referred to as a \emph{constellation} of shape $v_1,\ldots,v_k$, and serves as a multi-dimensional generalisation of the concept of an arithmetic progression.

In \cite{gt-primes}, Green and the first author obtained a version of Szemer\'edi's theorem relative to the primes (generalising a previous result of Green \cite{green} in the $k=3$ case):

\begin{theorem}[Szemer\'edi's theorem in the primes]\label{gt-thm}\cite{gt-primes}  Let $k \geq 1$ be a natural number, and let $\delta > 0$.  Then if $N$ is sufficiently large, every subset $A$ of ${\mathcal P} \cap [N]$ of cardinality $|A| \geq \delta |{\mathcal P} \cap [N]|$, where ${\mathcal P} = \{2,3,5,\ldots\}$ are the set of primes, contains an arithmetic progression $a,a+r,\ldots,a+(k-1)r$ of length $k$, where $a \in \Z$ and $r$ is a positive integer.
\end{theorem}

Among other things, Theorem \ref{gt-thm} implies that the primes ${\mathcal P}$ contain arbitrarily long arithmetic progressions.

The purpose of this paper is to similarly obtain a relative version of the multi-dimensional Szemer\'edi theorem in the primes, first conjectured in \cite[\S 12]{gauss}:

\begin{theorem}[Multidimensional Szemer\'edi theorem in the primes]\label{main}  Let $d \geq 1$ be a natural number, let $v_1,\ldots,v_k$ be elements of $\Z^d$, and let $\delta > 0$.  Then if $N$ is sufficiently large, every subset $A$ of $({\mathcal P} \cap [N])^d$ of cardinality $|A| \geq \delta |{\mathcal P} \cap [N]|^d$ contains a set of the form $a+rv_1,\ldots,a+rv_k$, where $a \in \Z^d$ and $r$ is a positive integer.
\end{theorem}

Theorem \ref{main} was previously obtained in \cite{cook} in the special case that $v_1,\ldots,v_k$ were in general position in the sense that no two of the $v_1,\ldots,v_k$ were in a common coordinate hyperplane.  Simultaneously\footnote{Note added in proof: subsequently to the initial release of this paper and of \cite{cookt}, a simplified proof of Theorem \ref{main} was given in \cite{fox}, which is similar to the arguments in this paper, except that the use of the weighted Furstenberg correspondence principle has been replaced by a simpler argument involving a weighted version of the Varnavides averaging trick \cite{var}.} with our paper, an alternate proof of this result has also been obtained by Cook, Magyar and Titichetrakun \cite{cookt}; see the discussion below for a comparison between the two arguments.  

Our proof of this theorem differs slightly from the proof of Theorem \ref{gt-thm}.  To prove the latter theorem in \cite{gt-primes}, sieve theory was used to construct a weight $\nu: \Z/N'\Z \to \R^+$ which was concentrated on a slight variant $\{ n \in [N'']: Wn+b \in {\mathcal P} \}$ of the primes (for suitable choices of parameters $N'', W, b$, with $[N'']$ embedded into $\Z/N'\Z$), which obeyed a finite number of pseudorandomness conditions (referred to in that paper as the \emph{linear forms conditions} and \emph{correlation conditions}), and then used methods\footnote{In subsequent work by Gowers \cite{gowers} and independently Reingold-Trevisan-Tulsiani-Vadhan \cite{reingold}, a much simpler proof of the transference principle, relying primarily on the Hahn-Banach theorem, was obtained.} inspired by ergodic theory to establish a \emph{transference principle} which allowed one to deduce from Theorem \ref{sz-thm} a relative version of that theorem, in which $A$ had positive density relative to $\nu$ instead of uniform measure.  See \cite{gt-primes} for details.  Very recently in \cite{cfz}, it has been shown that the correlation condition can be omitted in the hypotheses for \cite{gt-primes}, although some portion of the linear forms condition is still necessary.

A similar strategy was employed in \cite{gauss} to obtain an analogue of Theorem \ref{main} in which $d=2$, and the Cartesian square ${\mathcal P}^2$ of the set ${\mathcal P}$ of (rational) primes was replaced by the set ${\mathcal P}[i]$ of \emph{Gaussian primes}.  The main point here was that one could similarly use sieve theory to construct a two-dimensional weight function $\nu$ concentrated on (a slight variant of) the Gaussian primes which obeyed pseudorandomness conditions analogous to those in the rational case.  There were some slight technical issues relating to the symmetries of the set of Gaussian primes with respect to reflection around the coordinate axes, but these turned out not to seriously impact the required pseudorandomness properties for $\nu$.  (See also \cite{tz}, \cite{gt-selberg}, \cite{hoang}, \cite{zhou} for several other contexts in which the transference principle has been deployed.)

However, it turns out to be difficult to adopt this strategy directly for Theorem \ref{main}, even in the $d=2$ case, unless one assumes a general position hypothesis as in \cite{cook}.  The problem (which was already observed in \cite[\S 12]{gauss}, see also \cite{cook}) is that, in contrast to the Gaussian primes ${\mathcal P}[i]$, the Cartesian square ${\mathcal P}^2$ of the rational primes have a number of strong self-correlations that prevent the construction of a two-dimensional weight function $\nu$ that was both concentrated on ${\mathcal P}^2$ and obeyed pseudorandomness conditions analogous to those used in \cite{gt-primes} or \cite{gauss} (or even the reduced conditions in \cite{cfz}).  More specifically, given four integers $a,b,c,d$ drawn at random from a suitable probability distribution, there is a very strong correlation between the events $(a,c) \in {\mathcal P}^2$, $(b,c) \in {\mathcal P}^2$, $(a,d) \in {\mathcal P}^2$, $(b,d) \in {\mathcal P}^2$, since any three of these events imply the fourth.  This correlation (which is not present at all in the case of Gaussian primes) appears to defeat any attempt to directly mimic the arguments from \cite{gt-primes} or \cite{gauss}.  However, simultaneous with our work, Cook, Magyar, and Titichetrakun \cite{cookt} have found a way around this difficulty by obtaining a more general version of the relative hypergraph removal lemma established in \cite{gauss}, in which the hypergraph considered is no longer assumed to be $k$-uniform for a fixed $k$, and now has weights attached to many levels of uniformity in the hypergraph, rather than solely at the $k$-uniform level. The very recent results in \cite{cfz} on relaxing the hypotheses needed for the transference principle may also be useful in this regard, although some modification of their approach would still be needed to deal with the obstructions mentioned above.

In view of these obstructions, we adopt a slightly different strategy, based on an observation in an unpublished note \cite{corr} of the first author that one could rather easily obtain a relative version of Szemer\'edi's theorem (Theorem \ref{sz-thm}) provided that the weight $\nu$ used obeyed an \emph{infinite} number of linear forms conditions, as opposed to finitely many linear forms conditions combined with a correlation condition as in \cite{gt-primes} or \cite{gauss}, by obtaining a weighted generalisation of the Furstenberg correspondence principle \cite{furstenberg} connecting combinatorial theorems such as Theorem \ref{sz-thm} with multiple recurrence theorems in ergodic theory.  At the time, this observation was of limited use since no constructions were known of weights $\nu$ concentrated on the primes which obeyed an infinite number of linear forms conditions.  (The sieve weights used in \cite{gt-primes} are only known to obey a finite number of linear forms conditions, with the number of such conditions depending on the level of the sieve employed.)  However, thanks to the (rather lengthy) arguments by Green and the authors \cite{linprimes}, \cite{mobius}, \cite{gtz}, it is now known that the primes themselves (or more precisely, a certain weight function concentrated on a rescaled version of the primes) obey infinitely many linear forms conditions, thus making the arguments in \cite{corr} usable for applications to the primes.   It turns out that this strategy not only works to give another proof\footnote{Note though that the current proof of the results in \cite{linprimes} relies heavily on the same machinery that was used in \cite{gt-primes} to prove Theorem \ref{gt-thm}, so this is not a genuinely independent new proof of that theorem.  However, if a new proof of the results in \cite{linprimes} was discovered that avoided the methods of \cite{gt-primes}, then by combining that proof with the arguments in this paper one would be able to also obtain a new proof of Theorem \ref{gt-thm}.} of Theorem \ref{gt-thm}, but extends without much additional difficulty to the multi-dimensional case, and in particular will be used here to establish Theorem \ref{main}.  This yields an argument which is shorter than that in \cite{cookt}, but is significantly less self-contained due to its heavy reliance on the work in \cite{linprimes}, \cite{mobius}, \cite{gtz}.  

The authors thank David Conlon, Le Thai Hoang and Akos Magyar for some useful comments, and to the anonymous referee for many helpful suggestions and corrections.

\section{The Furstenberg correspondence principle}

The original proof in \cite{fk0} of the multi-dimensional Szemer\'edi theorem (Theorem \ref{fk-thm}) is deduced from the following recurrence theorem:

\begin{definition}[$\Z^d$-system]  Let $d \geq 1$ be a natural number.  A \emph{$\Z^d$-system} is a quadruplet $(X, {\mathcal B}, \mu, (T_h)_{h \in \Z^d})$, where $(X, {\mathcal B}, \mu)$ is a probability space (thus ${\mathcal B}$ is a $\sigma$-algebra on $X$ and $\mu$ is a countably additive probability measure on ${\mathcal B}$) and for each $h \in \Z^d$, $T_h: X \to X$ is a measure-preserving bijection with $T_{h+h'} = T_h T_{h'}$ for all $h,h' \in \Z^d$.
\end{definition}

\begin{theorem}[Multiple recurrence theorem]\label{mrt}\cite[Theorem A]{fk0} Let $(X, {\mathcal B}, \mu, (T_h)_{h \in \Z^d})$ be a $\Z^d$-system for some $d \geq 1$, let $E$ be an element of ${\mathcal B}$ with $\mu(E)>0$, and let $v_1,\ldots,v_k \in \Z^d$.  Then there exists a natural number $r>0$ such that
$$ \mu( T_{rv_1} E \cap \ldots \cap T_{rv_k} E ) > 0.$$
\end{theorem}

We will not reprove Theorem \ref{mrt} here, but instead use it as a ``black box''.

The deduction of Theorem \ref{fk-thm} from Theorem \ref{mrt} is an immediate consequence of the \emph{Furstenberg correspondence principle} \cite{furstenberg} for the group $\Z^d$, which we now state and prove.

\begin{proposition}[Unweighted Furstenberg correspondence principle]\label{fcp}  Let $d \geq 1$ be a fixed natural number.  Let $N = N_n$ be a sequence of natural numbers going to infinity, and write $[N_n]$ for the set $\{ m \in \Z: 1\leq m \leq N_n\}$.  For each $n$, let $A_n$ be a subset of $[N_n]^d$ such that one has the lower bound
\begin{equation}\label{lions}
 \limsup_{n \to \infty} \E_{a \in [N_n]^d} 1_{A_n}(a) \geq \delta
\end{equation}
 on the upper density of $A_n$ for some $\delta > 0$, where we use the averaging notation $\E_{a \in S} f(a) := \frac{1}{|S|} \sum_{a \in S} f(a)$ for any finite non-empty set $S$.  Then there exists a $\Z^d$-system $(X, {\mathcal B}, \mu, (T_h)_{h \in \Z^d})$ and a set $E \in {\mathcal B}$ such that
$$ \mu(E) \geq \delta$$
and such that whenever $h_1,\ldots,h_k \in \Z^d$ are a finite number of shifts such that
$$ \mu( T_{h_1} E \cap \ldots \cap T_{h_k} E) > 0$$
then for a sequence of $n$ tending to infinity, there exists $a \in [N_n]^d$ such that
$$ a+h_1, \ldots, a+h_k \in A_n.$$
\end{proposition}

Indeed, if Theorem \ref{fk-thm} fails, then one could find $\delta > 0$ and $v_1,\ldots,v_k \in \Z^d$, a sequence $N_n$ going to infinity, and subsets $A_n$ of $[N_n]^d$ obeying \eqref{lions}, but such that $A_n$ contained no constellation of the form $a+rv_1,\ldots,a+rv_k$ for $a \in \Z^d$ and $r > 0$; applying Proposition \ref{fcp} followed by Theorem \ref{mrt} gives the desired contradiction.

\begin{proof}   By passing to a subsequence of $n$ (and relabeling $n$) we may assume that the limiting density of $A_n$ exists, thus
\begin{equation}\label{nd}
 \lim_{n \to \infty} \E_{a \in [N_n]^d} 1_{A_n}(a) \geq \delta.
\end{equation}

Let 
$$X := 2^{\Z^d} := \{ B: B \subset \Z^d\}$$
be the space of all subsets $B$ of $\Z^d$.  By Tychonoff's theorem, this is a compact Hausdorff topological space, endowed with the product $\sigma$-algebra ${\mathcal B}$ (that is to say, the $\sigma$-algebra generated by the basic cylinder sets $\{ B \in X: b \in B \}$ for $b \in \Z^d$).  This space has a continuous action $(T_h)_{h \in \Z^d}$ of $\Z^d$, defined by
$$ T_h B := B+h$$
for all $h \in \Z^d$ and $B \in X$.  We also define $E \in {\mathcal B}$ to be the basic cylinder set
\begin{equation}\label{e-def}
 E := \{ B \in X: 0 \in B \};
 \end{equation}
note that this is a clopen subset of $X$.

It remains to construct a shift-invariant probability measure $\mu$ on $X$ with the required properties.  This will be constructed as the weak limit of (a subsequence of) a sequence $\mu_n$ of discrete probability measures, defined as follows.  For each $n$ and $a \in [N_n]^d$, we let $B_{a,n} \in X$ be the set
$$ B_{a,n} := \{ b \in \Z^d: a+b \in A_n \}$$
(thus $B_{a,n}$ is the shift of $A_n$ by $-a$) and then set
$$ \mu_n := \E_{a \in [N_n]^d} \delta_{B_{a,n}},$$
where $\delta_{B_{a,n}}$ is the Dirac measure on $X$ at the point $B_{a,n}$.  Then $\mu_n$ is a probability measure; indeed, it is the law of the random set\footnote{Here, by ``random set'', we mean a set that is also a random variable, although the events $i \in B_{a,n}$ for different $i$ are not indepenent for this particular set-valued random variable.} $B_{a,n}$ when $a$ is drawn uniformly at random from $[N_n]^d$.

We now form a limit point $\mu$ of the probability measures $\mu_n$.  One could invoke Prokhorov's theorem at this point to obtain a subsequence $\mu_{n_j}$ that converges in the vague topology to a limit probability measure $\mu$, but for minor technical reasons (which are not relevant for this particular argument, but will become significant in the proof of Proposition \ref{fcp-weight} below) it is convenient to proceed slightly differently using Banach limits\footnote{Note that the original proof \cite{furstenberg} of the correspondence principle also used Banach limits.}.  Let $p\!-\!\lim = p\!-\!\lim_{n \to \infty}: \ell^\infty(\N) \to \R$ be a Banach limit functional, that is to say a linear functional extending the limit functional $\lim$ on convergent sequences such that
$$ \liminf_{n \to \infty} x_n \leq p\!-\!\lim_{n \to \infty} x_n \leq \limsup_{n \to \infty} x_n$$
for all bounded sequences $x_n$.  Such a functional can be easily constructed by the Hahn-Banach theorem, or by using a non-principal ultrafilter $p \in \beta \N \backslash \N$ on the natural numbers.  We then define the measure $\mu$ to be the unique Radon measure such that
\begin{equation}\label{fmu}
 \int_X f\ d\mu := p\!-\!\lim_{n \to \infty} \int_X f\ d\mu_n
 \end{equation}
for any continuous function $f: X \to \R$; in particular,
\begin{equation}\label{muf}
\mu(F) = p\!-\!\lim_{n \to \infty} \mu_n(F)
\end{equation}
for any clopen set $F$.  From the Riesz representation theorem we see that $\mu$ is well-defined, and as each of the $\mu_n$ are probability measures, we see that $\mu$ is also.  

From construction we have
\begin{align*}
\mu_n(E) &= \E_{a \in [N_n]^d} 1_{B_{a,n}}(0)\\
&= \E_{a \in [N_n]^d} 1_{A_n}(a)
\end{align*}
and hence from \eqref{nd} and \eqref{muf} we have
$$ \mu(E) \geq \delta$$
as claimed.

Now we verify the shift invariance of $\mu$.  Observe that the pushforward $(T_h)_* \mu_n$ of $\mu_n$ by a shift $T_h$ with $h \in \Z^d$ is given by
\begin{align*}
(T_h)_* \mu_n &= \E_{a \in [N_n]^d} \delta_{B_{a,n}+h} \\
&= \E_{a \in [N_n]^d} \delta_{B_{a-h,n}}\\
&= \E_{a \in [N_n]^d - h} \delta_{B_{a,n}};
\end{align*}
in particular, as the symmetric difference of $[N_n]^d+h$ and $[N_n]^d$ has cardinality $o(N_n^d)$ for fixed $h$, we see that $(T_h)_* \mu_n - \mu_n$ converges to zero in total variation norm.  Integrating against a continuous test function and using \eqref{fmu}, we see that $(T_h)_* \mu = \mu$ for all $h$, giving the desired shift invariance.

Now suppose that $h_1,\ldots,h_k \in \Z^d$ are such that
$$ \mu( T_{h_1} E \cap \ldots \cap T_{h_k} E) > 0.$$
As $T_{h_1} E \cap \ldots \cap T_{h_k} E$ is clopen, we see from \eqref{muf} that
$$ \mu_n( T_{h_1} E \cap \ldots \cap T_{h_k} E) > 0$$
for sufficiently large $n$.  But the expression $\mu_n( T_{h_1} E \cap \ldots \cap T_{h_k} E)$ can be computed as
$$
\E_{a \in [N_n]^d} 1_{T_{h_1} E}(B_{a,n}) \ldots 1_{T_{h_k} E}(B_{a,n})
= 
\E_{a \in [N_n]^d} 1_{A_n}(a+h_1) \ldots 1_{A_n}(a+h_k)$$
and so there exists $a \in [N_n]^d$ such that $a+h_1,\ldots,a+h_k \in A_n$, as desired.
\end{proof}

\begin{remark}\label{stronger}  The above argument in fact shows that the conclusion of Proposition \ref{fcp} can be strengthened, namely that
$$ \limsup_{n \to \infty} \E_{a \in [N_n]^d} 1_{A_n}(a+h_1) \ldots 1_{A_n}(a+h_k) \geq \mu( T_{h_1} E \cap \ldots \cap T_{h_k} E)$$
for any $h_1,\ldots,h_k \in \Z^d$, and similarly if one or more of the $A_n$ (and the respective copies of $E$) are replaced by their complements.
\end{remark}

Now we give a weighted version of the above principle, which is one of the new contributions of the paper, and provides a new way to study dense sets or functions relative to pseudorandom measures in settings in which the pseudorandomness is not strong enough for existing decomposition theorems to be useful.

\begin{proposition}[Weighted Furstenberg correspondence principle]\label{fcp-weight}  Let $d \geq 1$ be a fixed dimension.  Let $N = N_n$ be a sequence of natural numbers going to infinity, and let $M_n$ be a sequence going to infinity such that $M_n = o_{n \to\infty}(N_n)$ (thus $M_n/N_n$ goes to zero as $n \to \infty$).  For brevity we abbreviate $X = o_{n \to \infty}(Y)$ by $X = o(Y)$ in the sequel.  For each $n$, suppose that we have functions $\nu_i = \nu_{i,n}: [N_n] \to R^+$ for $i=1,\ldots,d$ obeying the following \emph{linear forms condition}:
\begin{itemize}
\item (Linear forms condition) Let $m \geq 0$ be a fixed natural number, let $k_1,\ldots,k_d \geq 0$ be fixed natural numbers, and for each $1 \leq i \leq d$, let $\phi_{i,1},\ldots,\phi_{i,k_i}: \R^m \to \R$ be linear forms with integer coefficients (so, in particular, $\phi_{i,j}$ maps $\Z^m$ to $\Z$).  Suppose that for each $1 \leq i \leq d$, the linear forms $\phi_{i,j},\phi_{i,j'}$ with $1 \leq j < j' \leq k_i$ are pairwise distinct.  Then there exists a sequence $H_n = o(M_n)$ such that whenever $L_{1,n},\ldots,L_{m,n}$ are such that $H_n \leq L_{l,n} \leq M_n$ for all $1 \leq l \leq m$, one has
\begin{equation}\label{sam}
 \E_{a \in [N_n]^d} \E_{r \in \prod_{j=1}^m [L_{j,n}]} \prod_{i=1}^d \prod_{j=1}^{k_i} \nu_{i,n}(a_i + \phi_{i,j}(r) ) = 1+o(1) 
\end{equation}
where $a = (a_1,\ldots,a_d)$, and each $\nu_{i,n}$ is extended by zero outside of $[N_n]$ to $\Z$.
\end{itemize}
For each $n$, let $A_n$ be a subset of $[N_n]^d$ such that one has the lower bound
$$ \limsup_{n \to \infty} \E_{a \in [N_n]^d} 1_{A_n}(a) \prod_{i=1}^d \nu_i(a_i) \geq \delta$$
for some $\delta > 0$, where $a = (a_1,\ldots,a_d)$.  Then there exists a $\Z^d$-system $(X, {\mathcal B}, \mu, (T_h)_{h \in \Z^d})$ and a set $E \in {\mathcal B}$ such that
$$ \mu(E) \geq \delta$$
and such that whenever $h_1,\ldots,h_k \in \Z^d$ are a finite number of shifts such that
$$ \mu( T_{h_1} E \cap \ldots \cap T_{h_k} E) > 0$$
then for a sequence of $n$ tending to infinity, there exists $a \in [N_n]^d$ and $r \in [M_n]$ such that
$$ a+rh_1, \ldots, a+rh_k \in A_n.$$
\end{proposition}

One should view \eqref{sam} as encompassing an infinite number of linear forms conditions of the type introduced in \cite{gt-primes}, because the parameters $k_1,\ldots,k_d$ are allowed to be arbitrarily large.

We now prove this proposition. By passing to a subsequence, we may assume that
\begin{equation}\label{nd-2}
\liminf_{n \to \infty} \E_{a \in [N_n]^d} 1_{A_n}(a) \prod_{i=1}^d \nu_i(a_i) \geq \delta.
\end{equation}

We define the compact space $X$, the $\sigma$-algebra ${\mathcal B}$, the set $E$, and the shifts $(T_h)_{h \in \Z^d}$ exactly as in the proof of Proposition \ref{fcp}; the only difference in the argument is in the construction of the probability measure $\mu$, which involves an additional averaging over dilations (cf. the classical argument\footnote{A similar averaging was also employed in the very recent preprint \cite{fox}.} of Varnavides \cite{var}).  For each $a \in [N_n]^d$ and $r \in [M_n]$, we let $B_{a,r,n} \in X$ be the set
$$ B_{a,r,n} := \{ b \in \Z^d: a+rb \in A_n \}.$$
For every tuple $\Omega = (\Omega_1,\ldots,\Omega_d)$ of finite subsets of $\Omega_1,\ldots,\Omega_d$ of $\Z$, we then define the discrete non-negative finite measures $\mu_{\Omega,n}$ on $X$ by the formula
\begin{equation}\label{mosh}
 \mu_{\Omega,n} := \E_{a \in [N_n]^d} \E_{r \in [M_n]} \delta_{B_{a,r,n}} \prod_{i=1}^d \prod_{c_i \in \Omega_i} \nu_{i,n}( a_i + c_i r ),
\end{equation}
where we extend each $\nu_{i,n}$ by zero outside of $[N_n]$.  From \eqref{sam} we have
$$
\E_{a \in [N_n]^d} \E_{r \in [M_n]} \prod_{i=1}^d \prod_{c_i \in \Omega_i} \nu_{i,n}( a_i + c_i r ) = 1 + o(1)$$
and thus the $\mu_{\Omega,n}$ are asymptotically probability measures in the sense that
\begin{equation}\label{muon}
 \mu_{\Omega,n}(X) = 1 + o(1)
 \end{equation}
for each fixed choice of $\Omega$.

We now select a Banach limit functional $p\!-\!\lim = p\!-\!\lim_{n \to \infty}$ as in the previous argument, and define the limit measures $\mu_\Omega$ for each tuple $\Omega$ of finite subsets of $\Z$ by the formula
\begin{equation}\label{fmu-omega}
 \int_X f\ d\mu_\Omega = p\!-\!\lim_{n \to \infty}  \int_X f\ d\mu_{\Omega,n}
 \end{equation}
for any continuous function $f: X \to \R$; in particular,
\begin{equation}\label{muf-omega}
\mu_\Omega(F) = p\!-\!\lim_{n \to \infty}  \mu_{\Omega,n}(F)
\end{equation}
for any clopen set $F$.  From the Riesz representation theorem and \eqref{muon} we see that each $\mu_\Omega$ is well-defined and is a probability measure.  We make the technical remark that it will be useful in later arguments that the same limit functional $p\!-\!\lim$ is used to define each of the $\mu_\Omega$; this is the main reason why we work using limit functionals instead of simply invoking Prokhorov's theorem to construct each $\mu_\Omega$ separately.

Next, we define a second limit measure $\mu$ by the formula
$$
 \int_X f\ d\mu = p\!-\!\lim_{m \to \infty}  \int_X f\ d\mu_{([-m,m],\ldots,[-m,m])}$$
for all continuous functions $f: X \to \R$, where $[-m,m] := \{-m,\ldots,m\}$, and we have used a different variable here for the Banach limit to avoid confusion with the $n$ variable.  Again, from the Riesz representation theorem we see that $\mu$ is well-defined and is a probability measure.

\begin{remark} Informally, as a first approximation one can roughly view $\mu_{\Omega,n}$ as a random set $B_{a,r,n}$, where the tuple $(a,r) \in [N_n]^d \times [M_n]$ has been conditioned to the event that $a_i + c_i r$ lies in the support of $\nu_i$ (or the region where $\nu_i$ is concentrating on) for each $1 \leq i \leq r$ and $c_i \in \Omega_i$.  The measure $\mu$ is a double limit of this random variable in which $n$ has been sent to infinity and the $\Omega_i$ have been sent to all of $\Z$, but it is important that the limit in $n$ be taken first in order to obtain a useful probability measure $\mu$.
\end{remark}

To understand the properties of $\mu$, we need to use the linear forms condition to establish a crucial compatibility relationship between the $\mu_\Omega$.

\begin{proposition}[Compatibility]\label{compact}  Let $\Omega = (\Omega_1,\ldots,\Omega_d)$ and $\Omega' = (\Omega'_1,\ldots,\Omega'_d)$ be tuples of finite subsets of $\Z$ such that $\Omega_i \subset \Omega'_i$ for all $i$.  Let ${\mathcal B}_\Omega$ denote the finite sub-$\sigma$-algebra of ${\mathcal B}_{\Omega'}$ generated by the basic cylinder sets $\{ B \in X: b \in B \}$ for $b \in \prod_{i=1}^d \Omega_i$.  Then we have
$$ \mu_\Omega(F) = \mu_{\Omega'}(F)$$ 
for all $F \in {\mathcal B}_\Omega$.
\end{proposition}

\begin{proof}  By iteration, it suffices to establish the claim when $\Omega'_i = \Omega_i$ for all $i$ except for one $i_0 \in \{1,\ldots,d\}$, with $\Omega'_{i_0} = \Omega_{i_0} \cup \{c_*\}$
for some integer $c_*$ outside of $\Omega_{i_0}$.

It suffices to establish the claim for sets $F$ of the form
\begin{equation}\label{fbo}
 F = \{ B \in X: B_0 \subset B \}
 \end{equation}
for some subset $B_0$ of $\prod_{i=1}^d \Omega_i$, since from the inclusion-exclusion formula this then gives the claim for sets of the form
$$ F = \{ B \in X: B \cap \prod_{i=1}^d \Omega_i = B_0 \}$$
and all other sets in ${\mathcal B}_\Omega$ can be expressed as finite unions of sets of the above type.  Note that if $F$ is of the form \eqref{fbo}, then $F$ is clopen, so by \eqref{muf-omega}
$$\mu_\Omega(F) = p\!-\!\lim_{n \to \infty}  \mu_{\Omega,n}(F)$$
and
$$\mu_{\Omega'}(F) = p\!-\!\lim_{n \to \infty}  \mu_{\Omega',n}(F)$$
so it will suffice to show that
$$ \mu_{\Omega',n}(F) - \mu_{\Omega,n}(F) = o(1)$$
for fixed choices of $F, \Omega, \Omega'$.

From \eqref{mosh} and \eqref{fbo} we see that
\begin{equation}\label{liar}
 \mu_{\Omega,n}(F) := \E_{a \in [N_n]^d} \E_{r \in [M_n]} \left(\prod_{b \in B_0} 1_{A_n}(a+rb)\right) \prod_{i=1}^d \prod_{c_i \in \Omega_i} \nu_{i,n}( a_i + c_i r )
\end{equation}
and similarly
$$
 \mu_{\Omega',n}(F) := \E_{a \in [N_n]^d} \E_{r \in [M_n]} \left(\prod_{b \in B_0} 1_{A_n}(a+rb)\right) \prod_{i=1}^d \prod_{c_i \in \Omega'_i} \nu_{i,n}( a_i + c_i r ).
$$
From the hypothesis on $\Omega'$, we may thus write $\mu_{\Omega',n}(F) - \mu_{\Omega,n}(F)$ as
$$
\E_{a \in [N_n]^d} \E_{r \in [M_n]} (\prod_{b \in B_0} 1_{A_n}(a+rb)) \left(\prod_{i=1}^d \prod_{c_i \in \Omega_i} \nu_{i,n}( a_i + c_i r )\right)
(\nu_{i_0,n}(a_{i_0}+c_* r)-1).$$
Suppressing all explicit dependence on $n$, our task is now to show that
\begin{equation}\label{bando}
\E_{a \in [N]^d} \E_{r \in [M]} F(a,r) = o(1),
\end{equation}
where 
\begin{equation}\label{far-def} F(a,r) := \left(\prod_{b \in B_0} 1_{A}(a+rb)\right) \left(\prod_{i=1}^d \prod_{c_i \in \Omega_i} \nu_{i}( a_i + c_i r )\right)
(\nu_{i_0}(a_{i_0}+c_* r)-1).
\end{equation}
To verify this, we first make a technical observation

\begin{lemma} Let  $H = H_n = o(M)$.  Then we have
\begin{equation}\label{shift-2}
\E_{a \in [N]^d} \E_{r \in [M]} F(a+b,r+h) = \E_{a \in [N]^d} \E_{r \in [M]} F(a,r) + o(1)
\end{equation}
whenever $|b|, |h| \leq H$, uniformly in $b,h$.
\end{lemma}

\begin{proof} By the triangle inequality, it suffices to show that 
\begin{equation}\label{shift}
\sum_{(a,r) \in \Delta_H} |F(a,r)| = o(N^d M)
\end{equation}
where
$$ \Delta_H := ([-H,N+H]^d \times [-H,M+H]) \backslash ([H,N-H]^d \times [H,M-H])$$
is the $H$-neighbourhood of the boundary of $[N]^d \times [M]$.

Let $L = o(M)$ be a quantity with $2H \leq L \leq M$ to be chosen later.  To prove \eqref{shift}, we observe that
\begin{align*}
\sum_{(a,r) \in \Delta_H} |F(a,r)| &\ll L^{-2d-2} \sum_{a \in [N]^d} \sum_{r \in [M]} 1_{\Delta_{H+2L}}(a,r)
\sum_{b,b' \in [L]^d; s,s' \in [L]} |F(a+b-b', r+s-s')| \\
&\ll
L^{-2d-2} |\Delta_{H+2L}|^{1/2}
(\sum_{a \in [N]^d} \sum_{r \in [M]} (\sum_{b,b' \in [L]^d; s,s' \in [L]} |F(a+b-b', r+s-s')|)^2)^{1/2} \\
&\ll (L/M)^{1/2} N^d M 
(\E_{a \in [N]^d; r \in [M]; b,b',b'',b'' \in [L]^d; s,s',s'',s''' \in [L]} \\
&\quad F(a+b-b', r+s-s') F(a+b''-b''',r+s''-s'''))^{1/2} 
\end{align*}
where we have used the Cauchy-Schwarz inequality and the bound $|\Delta_{H+2L}| \ll (L/M) N^d M$; here and in the sequel we use the notation $X \ll Y$ or $Y \gg X$ to denote the estimate $X \leq CY$ where $C$ is independent of $n$.  Thus, to show \eqref{shift}, it will suffice to show that
$$
\E_{a \in [N]^d; r \in [M]; b,b',b'',b'' \in [L]^d; s,s',s'',s''' \in [L]} F(a+b-b', r+s-s') F(a+b''-b''',r+s''-s''') \ll 1 + o(1).$$
By \eqref{far-def}, the left-hand side may be bounded by
\begin{align*}
&\E_{a \in [N]^d; r \in [M]; b,b',b'',b'' \in [L]^d; s,s',s'',s''' \in [L]}\\
&\quad (\prod_{i=1}^d \prod_{c_i \in \Omega_i} \nu_{i}( a_i + b_i-b'_i+ c_i r + c_i (s-s')) \nu_{i}( a_i +b''_i-b'''_i+ c_i r + c_i (s''-s'''))) \times \\
&\quad \times (\nu_{i_0}(a_{i_0}+b_{i_0}-b'_{i_0}+c_* r+ c_*(s-s'))+1) (\nu_{i_0}(a_{i_0}+b''_{i_0}-b'''_{i_0}+c_* r+ c_* (s''-s'''))+1).
\end{align*}
By four applications of \eqref{sam}, this expression will be $4+o(1)$ if $L/M$ is sufficiently slowly decaying, and the claim \eqref{shift} (and hence \eqref{shift-2}) follows.  
\end{proof}

We return to the task of establishing \eqref{bando}.  We first introduce some additional averaging.  Let $H = o(M)$ be a quantity to be chosen later.  From \eqref{shift-2} we see that
$$
\E_{a \in [N]^d} \E_{r \in [M]} F(a-\sum_{b \in B} b h_b,r+ \sum_{b \in B} h_b) = \E_{a \in [N]^d} \E_{r \in [M]} F(a,r) + o(1)
$$
whenever $h_b \in [H]$ for $b \in B$, where $B := \prod_{i=1}^d \Omega_i$.  Averaging over the $h_b$, we conclude that
$$
\E_{a \in [N]^d} \E_{r \in [M]} \E_{h_B \in [H]^B} F(a-\sum_{b \in B} b h_b,r+ \sum_{b \in B} h_b) = \E_{a \in [N]^d} \E_{r \in [M]} F(a,r) + o(1)
$$
where $h_B := (h_b)_{b \in B}$, and so it will suffice to show that
\begin{equation}\label{sandal}
\E_{a \in [N]^d} \E_{r \in [M]} \E_{h_B \in [H]^B} F(a-\sum_{b \in B} b h_b,r+ \sum_{b \in B} h_b) = o(1).
\end{equation}
To explain why we introduce the additional variables $h_B$, observe that expression
\begin{equation}\label{fabo}
F(a-\sum_{b \in B} b h_b,r+ \sum_{b \in B} h_b)
\end{equation}
can be expanded as
\begin{align*}
&(\prod_{b' \in B_0} 1_{A}(a+rb' + \sum_{b \in B \backslash \{b'\}} (b'-b) h_b)) \times \\
&\quad \times (\prod_{i=1}^d \prod_{c_i \in \Omega_i} \nu_{i}( a_i + c_i r + \sum_{b \in B: c_i \neq b_i} (c_i-b_i) h_b)) \times \\
&\quad \times (\nu_{i_0}-1)(a_{i_0}+c_* r + \sum_{b \in B} (c_* - b_{i_0}) h_b).
\end{align*}
The point is then that each of the factors $1_{A}(a+rb' + \sum_{b \in B} (b'-b) h_b)$, $\nu_{i}( a_i + c_i r + \sum_{b \in B} (c_i-b_i) h_b)$ only depends on a proper subset of the variables $h_b$ with $b \in B$, with only the final factor $(\nu_{i_0}-1)(a_{i_0}+c_* r + \sum_{b \in B} (c_* - b_{i_0}) h_b)$ being dependent on all of the components of $h$.  To emphasise this structure, we now write \eqref{fabo} as
$$
\prod_{B' \subset B} f_{B', a, r}(h_{B'})$$
where $h_{B'} := (h_b)_{b \in B'}$,
$$ f_{B,a,r}(h_B) := (\nu_{i_0}-1)(a_{i_0}+c_* r + \sum_{b \in B} (c_* - b_{i_0}) h_b),$$
and for each proper subset $B' \subsetneq B$ of $B$, we define
\begin{align*}
f_{B',a,r}(h_{B'}) &:= \left(\prod_{b' \in B_0: B' = B \backslash \{b'\}} 1_{A}(a+rb' + \sum_{b \in B'} (b'-b) h_b)\right) \times \\
&\quad \left(\prod_{i=1}^d \prod_{c_i \in \Omega_i: B' = \{ b \in B: c_i \neq b_i\}} \nu_{i}( a_i + c_i r + \sum_{b \in B'} (c_i-b_i) h_b)\right).
\end{align*}
Observe that for all such proper subsets $B' \subsetneq B$, we have
$$ |f_{B',a,r}(h_{B'})| \leq \nu_{B',a,r}(h_{B'})$$
where
$$
\nu_{B',a,r}(h_{B'}) :=
\prod_{i=1}^d \prod_{c_i \in \Omega_i: B' = \{ b \in B: c_i \neq b_i\}} \nu_{i}( a_i + c_i r + \sum_{b \in B} (c_i-b_i) h_b).$$
We can then rewrite \eqref{sandal} as the bound 
\begin{equation}\label{sandal-2}
\E_{a \in [N]^d} \E_{r \in [M]} \E_{h_B \in [H]^B} \prod_{B' \subset B} f_{B', a, r}(h_{B'}) = o(1).
\end{equation}

We now invoke a \emph{weighted generalised von Neumann theorem} from \cite[Corollary B.4]{gt-primes}, which asserts the inequality
$$ |\E_{h_B \in [H]^B} \prod_{B' \subset B} f_{B', a, r}(h_{B'})|
\leq \| f_{B,a,r}\|_{\Box^B(\nu)} \prod_{B' \subsetneq B} \| \nu_{B',a,r} \|_{\Box^{B'}(\nu)}^{1/2^{|B|-|B'|}}$$
where
\begin{equation}\label{boxb}
\| f \|_{\Box^{B'}(\nu)} := \left( \E_{h^{(0)}_{B'}, h^{(1)}_{B'} \in [H]^{B'}} 
(\prod_{\omega_{B'} \in \{0,1\}^{B'}} f(h^{(\omega_{B'})_{B'}}))
(\prod_{B'' \subsetneq B'} \prod_{\omega_{B''} \in \{0,1\}^{B''}} \nu_{B'',a,r}( h^{(\omega_{B''})}_{B''}))
\right)^{1/2^{|B'|}}
\end{equation}
for all $B' \subset B$ and functions $f: [H]^{B'} \to \R$, where $h^{(i)}_{B'} = (h^{(i)}_b)_{b \in B'}$ for $i=0,1$ and
$$ h^{(\omega_b)_{b \in B''}}_{B''} := (h^{(\omega_b)}_b)_{b \in B''}.$$
This inequality is established by a finite number of applications of the Cauchy-Schwarz inequality; see \cite[Corollary B.4]{gt-primes} for details.  For us, the purpose of this inequality is to eliminate all dependence on the set $A$, leaving behind expressions that depend only on the $\nu_i$.  Applying this inequality, we see that we can bound the left-hand side of \eqref{sandal-2} by
$$
\E_{a \in [N]^d} \E_{r \in [M]} \| f_{B,a,r}\|_{\Box^B(\nu)} \prod_{B' \subsetneq B} \| \nu_{B',a,r} \|_{\Box^{B'}(\nu)}^{1/2^{|B|-|B'|}}.$$
To prove \eqref{sandal-2}, it thus suffices by the H\"older inequality to establish the following bounds:

\begin{lemma} We have
\begin{equation}\label{hold-1}
\E_{a \in [N]^d} \E_{r \in [M]} \| f_{B,a,r}\|_{\Box^B(\nu)}^{2^{|B|}} = o(1)
\end{equation}
and
\begin{equation}\label{hold-2}
\E_{a \in [N]^d} \E_{r \in [M]} \| \nu_{B',a,r}\|_{\Box^{B'}(\nu)}^{2^{|B'|}} = 1 + o(1)
\end{equation}
for all $B' \subsetneq B$.
\end{lemma}

\begin{proof}
Both of these estimates will ultimately follow from the linear forms condition \eqref{sam}.  We first verify this for \eqref{hold-1}.  By \eqref{boxb}, we can expand the left-hand side of \eqref{hold-1} as
\begin{align*}
& \E_{a \in [N]^d; r \in [M]; h^{(0)}_B, h^{(1)}_B \in [H]^B} \\
&\quad (\prod_{\omega_{B} \in \{0,1\}^{B}} (\nu_{i_0}-1)(a_{i_0}+c_* r + \sum_{b \in B} (c_* - b_{i_0}) h_b^{(\omega_b)})) \times \\
&\quad \times (\prod_{B'' \subsetneq B} \prod_{\omega_{B''} \in \{0,1\}^{B''}} 
\prod_{i=1}^d \prod_{c_i \in \Omega_i: B'' = \{ b \in B: c_i \neq b_i\}} \nu_{i}( a_i + c_i r + \sum_{b \in B''} (c_i-b_i) h^{(\omega_b)}_b) ).
\end{align*}
We can expand this expression further as an alternating sum of $2^{|B|}$ expressions, each of the form \eqref{sam}.  The claim \eqref{hold-1} then follows from \eqref{sam} if $H/M$ is chosen to decay to zero sufficiently slowly, provided that we can show that for each $1 \leq i \leq d$, the linear forms
$$ (r,h^{(0)}_B,h^{(1)}_B) \mapsto c_i r + \sum_{b \in B''} (c_i-b_i) h^{(\omega_b)}_b)$$
for $B'' \subsetneq B$, $\omega_{B''} \in \{0,1\}^{B''}$, and $c_i \in \Omega_i$ with $B'' = \{b \in B: c_i \neq b_i\}$, and also (if $i=i_0$)
$$ (r,h^{(0)}_B,h^{(1)}_B) \mapsto c_* r + \sum_{b \in B} (c_* - b_{i_0}) h_b^{(\omega_b)}$$
for $\omega_B \in \{0,1\}^B$, are all distinct.  However, an inspection of these forms reveal that each form depends on a different set of variables $h^{(\omega_b)}_b$ with non-zero coefficients, and the claim follows.

In a similar spirit, the left-hand side of \eqref{hold-2} may be expanded as
$$ \E_{a \in [N]^d; r \in [M]; h^{(0)}_{B'}, h^{(1)}_{B'} \in [H]^{B'}}
\prod_{B'' \subseteq B'} \prod_{\omega_{B''} \in \{0,1\}^{B''}} 
\prod_{i=1}^d \prod_{c_i \in \Omega_i: B'' = \{ b \in B': c_i \neq b_i\}} \nu_{i}( a_i + c_i r + \sum_{b \in B''} (c_i-b_i) h^{(\omega_b)}_b) ,$$
and by using \eqref{sam} as before we see that this expression is $1+o(1)$ as required.  
\end{proof}

The proof of Proposition \ref{compact} is (finally) complete.
\end{proof}

From Proposition \ref{compact} and the construction of $\mu$, we obtain as a corollary that
\begin{equation}\label{muu}
 \mu(F) = \mu_{\Omega}(F)
 \end{equation}
whenever $\Omega = (\Omega_1,\ldots,\Omega_d)$ is a tuple of finite subsets of $\Z$, and $F$ is ${\mathcal B}_\Omega$-measurable (and thus clopen).  In particular, since the set $E$ given by \eqref{e-def} is $(\{0\},\ldots,\{0\})$-measurable, we have
$$ \mu(E) = \mu_{(\{0\},\ldots,\{0\})}(E).$$
Applying \eqref{muf-omega} and \eqref{mosh}, we conclude that
\begin{align*}
 \mu(E) &= p\!-\!\lim_{n \to \infty}   \E_{a \in [N_n]^d} \E_{r \in [M_n]} 1_E(B_{a,r,n}) \prod_{i=1}^d \nu_{i,n}( a_i )\\
 &= p\!-\!\lim_{n \to \infty}   \E_{a \in [N_n]^d} 1_{A_n}(a) \prod_{i=1}^d \nu_{i,n}( a_i )
\end{align*}
and hence by \eqref{nd-2} we have
$$ \mu(E) \geq \delta$$
as required.

In a similar vein, suppose that $h_1,\ldots,h_k \in \Z^d$ are such that
$$ \mu( T_{h_1} E \cap \ldots \cap T_{h_k} E) > 0.$$
We can find a tuple $\Omega = (\Omega_1,\ldots,\Omega_d)$ of finite subsets of $\Z$ such that $\prod_{i=1}^d \Omega_i$ contains $h_1,\ldots,h_k$.  From \eqref{muu} we thus have
$$ \mu_\Omega( T_{h_1} E \cap \ldots \cap T_{h_k} E) > 0.$$
and thus for \eqref{muf-omega} we have
$$ \mu_{\Omega,n}( T_{h_1} E \cap \ldots \cap T_{h_k} E) > 0$$
for arbitrarily large $n$.  However, from \eqref{mosh}, we see that this quantity can be expanded as
$$
\E_{a \in [N_n]^d} \E_{r \in [M_n]} \prod_{j=1}^k 1_{T_{h_j} E}(B_{a,r,n}) \prod_{i=1}^d \prod_{c_i \in \Omega_i} \nu_{i,n}( a_i + c_i r );$$
since $1_{T_{h_j} E}(B_{a,r,n}) = 1_{A_n}(a+rh_j)$, we conclude that there is $a \in [N_n]^d$ and $r \in [M_n]$ such that
$$ a+rh_1, \ldots, a+rh_k \in A_n,$$
as required.

To conclude the proof of Proposition \ref{fcp-weight}, it now suffices to show

\begin{lemma}[Translation invariance]  For every $h \in \Z^d$, one has $(T_h)_* \mu = \mu$.
\end{lemma}

\begin{proof}  From inner and outer regularity, and the fact that the topology of $X$ has a clopen base of basic cylinder sets, it suffices to show that
$$ \mu(T_h F) = \mu(F)$$
whenever $F$ is a boolean combination of finitely many basic clopen sets, and in particular lies in ${\mathcal B}_\Omega$ for some tuple $\Omega= (\Omega_1,\ldots,\Omega_d)$ of finite subsets of $\Z$.  By the inclusion-exclusion formula as before, we may take $F$ to be of the form \eqref{fbo} for some $B_0 \subset \prod_{i=1}^d \Omega_i$.
From \eqref{muu}, it thus suffices to show that
$$ \mu_{\Omega+h}(T_h F) =\mu_\Omega(F)$$
where $h = (h_1,\ldots,h_d)$ and $\Omega+h := (\Omega_1+h_1,\ldots,\Omega_d+h_d)$.  By \eqref{muf-omega}, it suffices to show that
\begin{equation}\label{load}
 \mu_{\Omega+h,n}(T_h F) = \mu_{\Omega,n}(F) + o(1).
\end{equation}
From \eqref{liar} we have
$$
\mu_{\Omega,n}(F) = \E_{a \in [N]^d} \E_{r \in [M]} F(a,r)$$
where $F(a,r)$ is defined by \eqref{far-def}.  A similar computation reveals that
$$
\mu_{\Omega+h,n}(T_h F) = \E_{a \in [N]^d} \E_{r \in [M]} F(a+hr,r).$$
The claim \eqref{load} then follows from \eqref{shift}.
\end{proof}

The proof of Proposition \ref{fcp-weight} is now complete.

By combining Proposition \ref{fcp-weight} with Theorem \ref{mrt}, and making the trivial but important observation that the constellation $a+rv_1,\ldots,a+rv_k$ has a dilation symmetry, in the sense that if $a+rh_1,\ldots,a+rh_k \in A$ for some $h_1 = sv_1,\ldots,h_k=sv_k$ and some natural numbers $r,s$, then $A$ contains a constellation of the form $a+r' v_1,\ldots,a+r' v_k$ where $r' := rs$, we obtain a relative version of Theorem \ref{fk-thm}:

\begin{theorem}[Relative multiple Szemer\'edi theorem]\label{rmst}  Let $d \geq 1$ be a fixed dimension.  Let $N = N_n$ be a sequence of natural numbers going to infinity, and let $M_n$ be a sequence going to infinity such that $M_n = o_{n \to\infty}(N_n)$.  For each $n$, suppose that we have functions $\nu_i = \nu_{i,n}: [N_n] \to R^+$ for $i=1,\ldots,d$ obeying the linear forms condition from Proposition \ref{fcp-weight}.  Let $\delta>0$ and $v_1,\ldots,v_k \in \Z^d$.  Then if $n$ is sufficiently large, any subset $A_n$ of $[N_n]^d$ with
$$ \E_{a \in [N_n]^d} 1_{A_n}(a) \prod_{i=1}^d \nu_{i,n}(a_i) \geq \delta$$
contains a constellation of the form $a+r'v_1,\ldots,a+r'v_k$ with $a \in \Z^d$ and $r'>0$ a natural number.
\end{theorem}

Indeed, if this theorem fails, then by applying Proposition \ref{fcp-weight} to a sequence of counterexamples one can use Theorem \ref{mrt} to obtain a contradiction, just as in the proof of Theorem \ref{fk-thm} via Proposition \ref{fcp}.

\begin{remark}\label{stronger-weight} As in Remark \ref{stronger}, one can strengthen the conclusion of Proposition \ref{fcp-weight}, namely that
$$ \limsup_{n \to \infty} \mu_{\Omega,n}(T_{h_1} E \cap \ldots \cap T_{h_k} E ) \geq \mu( T_{h_1} E \cap \ldots \cap T_{h_k} E)$$
for any $h_1,\ldots,h_k \in \Z^d$ and any $\Omega_1,\ldots,\Omega_d$ whose Cartesian product contains $h_1,\ldots,h_k$, and similarly if one or more of the $A_n$ (and the respective copies of $E$) are replaced by their complements.
\end{remark}

\section{Proof of Theorem \ref{main}}

With the aid of the relative multiple Szemer\'edi theorem (Theorem \ref{rmst}), we are now ready to prove Theorem \ref{main}.  Suppose for contradiction that this theorem fails, thus we can find $d \geq 1$, $\delta>0$, and $v_1,\ldots,v_k \in \Z^d$, a sequence $N_n$ going to infinity as $n \to \infty$, and subsets $A_n$ of $({\mathcal P} \cap [N_n])^d$ such that
$$ |A_n| \geq \delta |{\mathcal P} \cap [N_n]|^d,$$
but such that no $A_n$ contains a constellation of the form $a+rv_1,\ldots,a+rv_k$ with $a \in \Z^k$ and $r$ a positive integer.  By deleting a small number of elements from $A_n$ (and reducing $\delta$ accordingly), we may assume that $A_n$ lies in $[\delta' N_n,N_n]^d$ for some $0 < \delta' \leq 1/2$ that can depend on $\delta$ but is independent of $n$.

From the prime number theorem, we have
$$ |A_n| \gg \frac{N_n^d}{\log^d N_n}$$
where we allow the implied constants in the $\gg$ notation to depend on $d, \delta, \eps, v_1,\ldots,v_k$, but not on $n$.  

We now drop the explicit dependence on $n$, thus for instance writing $A = A_n$ and $N = N_n$.  It is tempting to now apply Theorem \ref{rmst}, but the primes contain some local structure (in particular, they largely avoid certain residue classes) which prevent one from directly constructing a collection $\nu_{i,n}$ of weights obeying the required pseudorandomness conditions.  To deal with this issue, we use the ``$W$-trick'' from \cite{green-roth}, \cite{gt-primes}.  We let $w = w_n$ be a quantity growing (slowly) to infinity (with the precise choice of $w$ to be determined later), and let 
$$ W = W_n := \prod_{p \leq w} p$$
be the product of all the primes up to $w$.  If $w=w_n$ grows sufficiently slowly with $n$, then for sufficiently large $n$, all primes in $[\delta' N_n,N_n]$ are coprime to $W$, and thus lie in one of the $\phi(W)$ residue classes $b \mod W$ coprime to $W$, where $\phi(W)$ is the Euler totient function.  Applying the pigeonhole principle, we can then find for each sufficiently large $n$, residue classes $b_1 \mod W, \ldots, b_d \mod W$ (depending on $n$) that are coprime to $W$, and are such that
$$ |\{ a \in A: a_i = b_i \mod W \hbox{ for all } 1 \leq i \leq d \}| \gg \frac{N^d}{\phi(W)^d \log^d N}.$$
If we then set
$$ N' = N'_n := (1-\delta') N / W$$
and
$$ A' = A'_n := \{ a \in [N']^d: W(a + \lfloor \delta' N/W\rfloor) + b \in A \}$$
where $b := (b_1,\ldots,b_d)$, then we have 
\begin{equation}\label{alo}
 |A'| \gg \frac{W^d (N')^d}{\phi(W)^d \log^d N}.
\end{equation}
Also, we have 
\begin{equation}\label{nprod}
A' \subset \prod_{i=1}^d {\mathcal P}'_i
\end{equation}
where
$$ {\mathcal P}'_i := \{ a_i \in [N']: W(a_i + \lfloor \delta' N/W\rfloor) + b_i \in {\mathcal P} \}.$$
Also, since $A$ contains no constellations of the form $a+rv_1,\ldots,a+rv_k$ with $a \in \Z^d$ and $r$ a positive integer, the same claim is true for $A'$ (here we again take advantage of the dilation symmetry in the constellation $a+rv_1,\ldots,a+rv_k$). 

We now introduce the weights $\nu_i = \nu_{i,n}: [N'] \to \R^+$ for $i=1,\ldots,d$ by the formula
$$ \nu_i(a_i) := \frac{\phi(W)}{W} (\log N) 1_{{\mathcal P}'_i}(a_i),$$
then from \eqref{alo} we have
$$ \E_{a \in [N']^d} 1_{A'}(a) \prod_{i=1}^d \nu_i(a_i) \gg 1.$$
We can then apply Theorem \ref{rmst} to obtain a contradiction, as soon as we can find a sequence $M_n = o(N'_n)$ such that the pseudorandomness hypotheses in Proposition \ref{fcp-weight} are obeyed.

It suffices to show that for every collection of natural numbers $m, k_1,\ldots,k_d \geq 0$ and linear forms $\phi_{i,1},\ldots,\phi_{i,k_i}: \R^m \to \R$ with integer coefficients with the $\phi_{i,1}, \ldots, \phi_{i,k_i}$ pairwise distinct for each $i$, and any $\eps>0$ there exists a $0 < \kappa \leq \eps$ such that if $w>\frac{1}{\kappa}$ is fixed independently of $n$ (which makes $W$ fixed also), $b_1,\ldots,b_d \mod W$ are coprime to $W$, 
$\lambda>0$ is independent of $n$ and
\begin{equation}\label{lk}
 \lambda \kappa N'_n \leq L_{l,n} \leq \kappa N'_n
 \end{equation}
for all $1 \leq l \leq m$, and $b_1,\ldots,b_d \mod W$ are any residue classes coprime to $W$, then one has
\begin{equation}\label{dope}
| \E_{a \in [N'_n]^d} \E_{r \in \prod_{j=1}^m [L_{j,n}]} \prod_{i=1}^d \prod_{j=1}^{k_i} \nu_{i,n}(a_i + \phi_{i,j}(r) ) - 1 | \leq \eps + o(1)
\end{equation}
where the decay rate in the $o(1)$ term is allowed to depend on all the given data $m,k_i,\phi_{i,j}, \eps, w, b_1,\ldots,b_d, \lambda$ (but is uniform in the $L_{1,n},\ldots,L_{m,n}$), so in particular
\begin{equation}\label{dope-2}
| \E_{a \in [N'_n]^d} \E_{r \in \prod_{j=1}^m [L_{j,n}]} \prod_{i=1}^d \prod_{j=1}^{k_i} \nu_{i,n}(a_i + \phi_{i,j}(r) ) - 1 | \leq 2\eps 
\end{equation}
whenever $n$ is sufficiently large depending on $m,k_i,\phi_{i,j}, \eps, w, b_1,\ldots,b_d, \lambda$.  Indeed, as the number of choices for $m,k_i,\phi_{i,j},w,b_i$ is countable (and the number of choices for $\eps,\lambda$ can be restricted to a countable set such as reciprocals of the natural numbers), one can then construct a sequence $M_n$ with $M_n/N'_n$ decaying (slowly) to zero as $n \to \infty$ and a sequence $w_n$ growing (slowly) to infinity as $n \to \infty$ with the properties required for Proposition \ref{fcp-weight} by a standard diagonalisation argument.  More precisely, if one countably enumerates the possible values of the $m,k_i,\phi_{i,j},b_i,\eps,\lambda$ with $\eps,\lambda$ the reciprocal of natural numbers, then for each natural number $M$, one can find $w^{(M)} \geq M$, $0 < \kappa^{(M)} \leq 1/M$ and $n^{(M)} $ such that \eqref{dope-2} holds for $w = w^{(M)}$, $\kappa = \kappa^{(M)}$, all $n \geq n^{(M)}$, any of the first $M$ choices of $m,k_i,\phi_{i,j}, \eps, w, b_1,\ldots,b_d, \lambda$ in this enumeration, and any $L_{1,n},\ldots,L_{m,n}$ obeying \eqref{lk}.  One can arrange for the $n^{(M)}$ to be increasing in $M$.  If we then set $w_n := w^{(M)}$ and $H_n := \kappa^{(M)} N_n$ for $n^{(M)} \leq n < n^{(M+1)}$, we obtain the claim.

It remains to establish \eqref{dope}.  For this we use the main result in \cite{linprimes} as formulated in Theorem 5.1 of that paper (combined with 
the main results of \cite{mobius}, \cite{gtz}). 
We need to verify the conditions of Theorem 5.1; indeed observe that no two of the linear forms
$$(a,r) \mapsto a_i +  \phi_{i,j}(r)$$
with $i=1,\ldots,d$ and $j=1,\ldots,k_i$ are linearly dependent, and therefore the system of linear forms $(a_i +  \phi_{i,j}(r))_{1 \le i \le d, 1 \le j \le k_i }$ is of finite complexity.

\section{Further remarks}

By pursuing the strengthening of Proposition \ref{fcp-weight} outlined in Remark \ref{stronger-weight}, one can obtain the lower bound of
\begin{equation}\label{cat}
(c(\delta) - o_{n \to \infty}(1)) \frac{N_n^{d+1}}{\log^{\sum_{i=1}^d |\Omega_i|} N_n}
\end{equation}
for the number of pairs $(a,r) \in \Z^d \times \N$ with $a+rv_1,\ldots,a+rv_k \in A$, where $c(\delta)>0$ depends only on $d,v_1,\ldots,v_k,\delta$ and $\Omega_i$ is the projection of the set $\{v_1,\ldots,v_k\} \subset \R^d$ onto the $i^{\operatorname{th}}$ coordinate.  In the opposite direction, standard upper bound sieves such as the Selberg sieve\footnote{One could also use the main results of \cite{linprimes}, \cite{mobius}, \cite{gtz} here, although this would be overkill.} give an upper bound of the form
$$ O\left( \frac{N_n^{d+1}}{\log^{\sum_{i=1}^d |\Omega_i|} N_n} \right)$$
where the implied constant in the $O()$ notation depends only on $d,v_1,\ldots,v_k$; thus (as in previous applications of the transference principle) the argument gives the right order of magnitude for the number of constellations $a+rv_1,\ldots,a+rv_k$ to be found in the set $A$.  On the other hand, the decay rate of $o_{n \to \infty}(1)$ in \eqref{cat} is ineffective, primarily due to the appeal to the Siegel-Walfisz theorem in \cite{mobius} but also due to the use of nonstandard analysis in \cite{gtz}.  As such, our arguments do not provide an effective bound for the magnitude of the smallest value of $N_n$ for which a given constellation $a+rv_1,\ldots,a+rv_k$ is to appear.  However, the methods in \cite{cookt}, relying on a new weighted hypergraph removal lemma, should provide effective bounds (though, as is usually the case with results based on regularity lemmas, one expects the bounds obtained here to be tower-exponential or worse).

As observed in \cite{tz}, the methods in \cite{gt-primes} allow for the conclusion of Theorem \ref{gt-thm} to be strengthened, by restricting the shift parameter $r$ to be of size at most $N^c$ for any fixed constant $c>0$.  This is ultimately because the weights $\nu$ constructed by sieve-theoretic techniques can still obey pseudorandom conditions on relatively short intervals such as $[N,N+N^c]$, provided that one reduces the sieve level accordingly.  It may be that one could achieve similar strengthenings using the hypergraph methods from \cite{cook}.  If however one wishes to use the methods in this paper to obtain a similar restriction of the shift parameter $r$ in Theorem \ref{main}, one would need to establish a version of the main results in \cite{linprimes} on similarly short intervals $[N,N+N^c]$.  If one seeks such bounds for arbitrarily small $c$, it is likely that one will be forced to assume some unproven hypotheses such as the Generalised Riemann Hypothesis (or at least the Generalised Density Conjecture).  Some of these issues will be addressed in a forthcoming paper of Soundararajan and the second author.  If one can obtain such localised versions of the results in \cite{linprimes}, it may then become possible to obtain a polynomial generalisation of Theorem \ref{main} (in the spirit of the main results of \cite{tz}), by incorporating the methods from \cite{tz}, and by replacing the multi-dimensional Szemer\'edi theorem from \cite{fk0} with the multi-dimensional polynomial Szemer\'edi theorem from \cite{bl}.  We will not pursue these matters here.

The arguments given in this paper also apply (almost surely) in the case when ${\mathcal P}$ is replaced by a random subset ${\mathcal A}$ of the integers in which the events $n \in {\mathcal A}$ occur with probability $n^{-o(1)}$ and are jointly independent.  Of course, in that case the linear forms condition can be almost surely verified from the Chernoff inequality and the Borel-Cantelli lemma, without the need of the difficult results in \cite{linprimes}, \cite{mobius}, \cite{gtz}.  However, these results were already known in the random case, down to the sharp threshold of $n^{-1/(k-1)}$ for the density; see \cite{schacht} (together with closely related work in \cite{conlon}, as well as earlier work in \cite{klr} for the case of progressions of length three, see also \cite{hl}).

\end{document}